\documentclass{amsart}
\newtheorem{thm}{Theorem}[section]

\newtheorem{lem}[thm]{Lemma}
\newtheorem{prop}[thm]{Proposition}
\theoremstyle{definition}
\newtheorem{defn}[thm]{Definition}
\theoremstyle{remark}
\newtheorem{rem}[thm]{Remark}
\newtheorem*{ex}{Example}
\numberwithin{equation}{section}

\usepackage{color}
\usepackage[colorlinks]{hyperref}        
\usepackage{mathrsfs}	
\usepackage{tikz}
\usetikzlibrary{math}
\usepackage{subfig}

\DeclareMathOperator*{\conv}{conv}
\DeclareMathOperator*{\conc}{conc}

\newcommand{\R}{\mathbb{R}}

\newcommand{\N}{\mathbb{N}}
\newcommand{\Z}{\mathbb{Z}}

\newcommand{\e}{\varepsilon}

\newcommand{\TV}{\text{\rm Tot.Var.}}

\newcommand{\green}[1]{\textcolor{green}{#1}}

\newcommand{\red}[1]{\textcolor{red}{#1}}

\newcommand{\Qtrans}{Q^\mathrm{trans}}
\newcommand{\Qcubic}{Q^\mathrm{Bianchini}}

\newcommand{\Qbm}{Q^\mathrm{BM}}
\newcommand{\tcanc}{\mathtt T^{\mathrm{canc}}}
\newcommand{\sigmaent}{\sigma^{\mathrm{ent}}}
\newcommand{\sigmarh}{\sigma^{\mathrm{rh}}}
\newcommand{\xint}{\mathtt X^{\mathrm{int}}}
\newcommand{\tint}{\mathtt T^{\mathrm{int}}}
\newcommand{\W}{\mathcal{W}}
\newcommand{\jleft}{\mathcal J^{\rm left}}
\newcommand{\jright}{\mathcal J^{\rm right}}
\newcommand{\tcr}{\mathtt T^{\mathrm{cr}}}
\newcommand{\feff}{\mathtt f^{\rm eff}}
\newcommand{\fQ}{\mathfrak Q}

\makeatletter
\def\grd@save@target#1{%
  \def\grd@target{#1}}
\def\grd@save@start#1{%
  \def\grd@start{#1}}
\tikzset{
  grid with coordinates/.style={
    to path={%
      \pgfextra{%
        \edef\grd@@target{(\tikztotarget)}%
        \tikz@scan@one@point\grd@save@target\grd@@target\relax
        \edef\grd@@start{(\tikztostart)}%
        \tikz@scan@one@point\grd@save@start\grd@@start\relax
        \draw[minor help lines] (\tikztostart) grid (\tikztotarget);
        \draw[major help lines] (\tikztostart) grid (\tikztotarget);
        \grd@start
        \pgfmathsetmacro{\grd@xa}{\the\pgf@x/1cm}
        \pgfmathsetmacro{\grd@ya}{\the\pgf@y/1cm}
        \grd@target
        \pgfmathsetmacro{\grd@xb}{\the\pgf@x/1cm}
        \pgfmathsetmacro{\grd@yb}{\the\pgf@y/1cm}
        \pgfmathsetmacro{\grd@xc}{\grd@xa + \pgfkeysvalueof{/tikz/grid with coordinates/major step}}
        \pgfmathsetmacro{\grd@yc}{\grd@ya + \pgfkeysvalueof{/tikz/grid with coordinates/major step}}
        \foreach \x in {\grd@xa,\grd@xc,...,\grd@xb}
        \node[anchor=north] at (\x,\grd@ya) {\pgfmathprintnumber{\x}};
        \foreach \y in {\grd@ya,\grd@yc,...,\grd@yb}
        \node[anchor=east] at (\grd@xa,\y) {\pgfmathprintnumber{\y}};
      }
    }
  },
  minor help lines/.style={
    help lines,
    step=\pgfkeysvalueof{/tikz/grid with coordinates/minor step}
  },
  major help lines/.style={
    help lines,
    line width=\pgfkeysvalueof{/tikz/grid with coordinates/major line width},
    step=\pgfkeysvalueof{/tikz/grid with coordinates/major step}
  },
  grid with coordinates/.cd,
  minor step/.initial=.2,           
  major step/.initial=1,            
  major line width/.initial=1pt,    
}
\makeatother

\begin{document}

%
%
%
%
%
%
%
%
%

\title[A ``Forward-in-time'' quadratic potential]
{A ``forward-in-time'' quadratic potential \\ for systems of conservation laws}

\author[]{Stefano Modena}

\address{%
Mathematisches Institut \\
Universit\"at Leipzig \\
Augustusplatz 10, 04109 Leipzig \\
Germany
}

\email{modena.stef@gmail.com, Stefano.Modena@math.uni-leipzig.de}

\thanks{The author would like to thank Prof. Stefano Bianchini for many helpful discussions about the topic of this paper.}
\thanks{The paper was submitted while the author was a Post-Doc at the MPI for Mathematics in the Sciences in Leipzig.}





\date{\today}
\dedicatory{Dedicated to Prof. Alberto Bressan on the occasion of his 60th birthday}

\begin{abstract}
A quadratic interaction potential $t \mapsto \Upsilon(t)$ for hyperbolic systems of conservation laws is constructed, whose value $\Upsilon(\bar t)$ at time $\bar t$ depends only on the present and the future profiles of the solution and not on the past ones. Such potential is used to bound the change of the speed of the waves at each interaction. 
\end{abstract}

\maketitle

\section{Introduction}
\label{sec:intro}

Let us consider the Cauchy problem for the system of conservation laws
\begin{equation}
\label{eq:cauchy}
\begin{cases}
u_t + F(u)_x = 0, \\
u(0, x) = \bar u(x),
\end{cases}
\end{equation}
where $F: \R^n \to \R^n$ is a smooth ($C^2$) function, which is assumed to be \emph{strictly hyperbolic}, i.e. its differential $DF(u)$ has $n$ distinct real eigenvalues in each point of its domain, and $\bar u$ is a $BV$ function with ``small'' total variation. 

As usual, we denote by $\lambda_1(u) < \cdots < \lambda_n(u)$ the eigenvalues of $DF(u)$ and by $r_1(u), \dots, r_n(u)$ its eigenvectors, $DF(u) r_k(u) = \lambda_k(u) r_k(u)$. The $d$-dimensional Lebesgue  measure on $\R^d$ is denoted by $\mathcal L^d$.

The $k$-th characteristic field is said to be \emph{genuinely non linear} (GNL) if $\lambda_k$ is strictly increasing in the direction of $r_k$, while it is said to be \emph{linearly degenerate} (LD) if $\lambda_k$ is constant in the direction of $r_k$. We will often refer to genuine non linearity and linear degeneracy as \emph{convexity properties} of the flux $F$. 

%
%
%
In what follows, by \emph{solution} to the Cauchy problem \eqref{eq:cauchy} we always mean \emph{vanishing viscosity solution}, i.e. a weak solution obtained as limit as $\mu \to 0$ of the viscosity approximations $u_t + F(u)_x = \mu u_{xx}$. The existence, uniqueness and stability of vanishing viscosity  solutions for general hyperbolic systems were established by S. Bianchini and A. Bressan in the fundamental paper \cite{BiaBre}. Many other notions of ``admissible solution'' have been proposed in the literature. Without entering into details, here we just recall that, for GNL/LD systems, vanishing viscosity solutions satisfy the \emph{Lax condition on shocks} (see, for instance, \cite[Section 8.3]{Daf}), while, for scalar conservation laws, the vanishing viscosity condition is equivalent to the \emph{entropy condition} (see, for instance \cite[Section 4.5]{Daf}).


\subsection{Strength and speed of waves}

In the analysis of solutions to hyperbolic conservation laws, the heart of the matter is to understand how the strength or the speed of waves can change across interactions. 

More precisely, let $u^\e$ be a piecewise constant approximate solution to the system \eqref{eq:cauchy}, constructed by means of the wavefront tracking algorithm or the Glimm scheme, $\e$ is the discretization parameter, see for instance \cite{Daf}; for our purpose it is not important whether one uses wavefront tracking approximations or the Glimm scheme. 

Roughly speaking, a \emph{wave} in $u^\e$ is a jump discontinuity line, its \emph{strength} is the size of the jump and its \emph{speed} is its slope in the $(t,x)$-plane. 

Let $\alpha, \alpha'$ be two waves in $u^\e$, as in Figure \ref{fig:interazione} and assume that, at time $\bar t$, $\alpha, \alpha'$ interact, i.e. they collide. The typical behavior of (approximate) solutions to conservation laws is that after the interaction the two waves $\alpha, \alpha'$ are somehow ``conserved'', but they undergo a change of their strength and/or speed.

\begin{figure}
\begin{tikzpicture}
\draw[->] (0,0.6) to (0,4);
\node[above] at (0,4) {$t$};
\draw[green] (1, 0.6) to (2,2);
\draw[red] (2,2) to (4.8,0.6);
\draw[red] (1.2, 4) to (2,2);
\draw[green] (2,2) to (6,4); 
\node[left, green] at (1.5, 1.4) {$\alpha$};
\node[right, red] at (3.3, 1.5) {$\alpha'$};
\node[left, red] at (1.6, 3.1) {$\alpha'$};
\node[right, green] at (4.2, 3) {$\alpha$};
\draw[dashed] (2,2) to (0,2);
\node[left] at (0,2) {$\bar t$};
\end{tikzpicture}
\caption{Typical interaction between two waves $\alpha, \alpha'$}
\label{fig:interazione}
\end{figure}
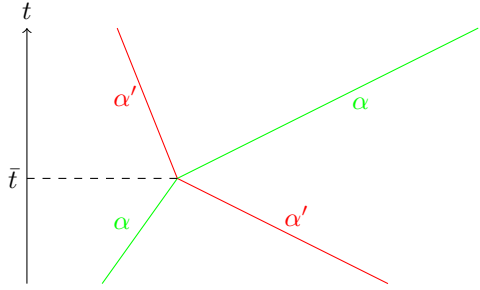

Understanding how the strength or the speed of waves can change across the interactions means that one is able to construct a suitable functional $t \mapsto Q(t)$, usually called \emph{interaction potential}, with the following properties:
\begin{enumerate}
\item $Q$ is uniformly bounded at time $0$: $Q(0) \leq C(F, \bar u)$, where $C(F, \bar u)$ is a constant which depends on the flux $F$ and the initial datum $\bar u$, but not on the discretization parameter $\e$;
\item $t \mapsto Q(t)$ is decreasing and, more precisely, if at time $\bar t$ two waves $\alpha, \alpha'$ interact, then 

\begin{equation}
\label{eq:potential}
\left.
\begin{split}
& |s(\bar t-, \alpha) - s(\bar t+, \alpha)|+ |s(\bar t-, \alpha') - s(\bar t+, \alpha')| 
\\[1em]
& \begin{split}
& |s(\bar t+, \alpha)||\sigma(\bar t-, \alpha) - \sigma(\bar t+, \alpha)| \\ 
& \qquad \qquad  + |s(\bar t+, \alpha')||\sigma(\bar t-, \alpha') - \sigma(\bar t+, \alpha')| 
\end{split}
\end{split}
\right\}
\leq Q(\bar t - ) - Q(\bar t+),
\end{equation}
where 
\begin{equation*}
\quad
\begin{split}
& s(\bar t \pm, \alpha) \text{ is the strength of $\alpha$ before and after the interaction time $\bar t$,} \\
& s(\bar t \pm, \alpha') \text{ is the strength of $\alpha'$ before and after the interaction time $\bar t$,} \\
& \sigma(\bar t \pm, \alpha) \text{ is the speed of $\alpha$ before and after the interaction time $\bar t$,} \\
& \sigma(\bar t \pm, \alpha) \text{ is the speed of $\alpha'$ before and after the interaction time $\bar t$.}
\end{split}
\end{equation*}
\end{enumerate}
The existence of a uniformly bounded potential with the property \eqref{eq:potential} is the fundamental tool in order to prove the existence of solutions to the system \eqref{eq:cauchy}, their uniqueness, their stability...

For instance, in the existence proof, a bound of the form \eqref{eq:potential} on the change of the strength of the waves (for the wavefront tracking) and on the change of the speed of the waves (for the Glimm scheme) is fundamental in order to construct a whole sequence of approximate solutions, to prove that such a sequence is pre-compact and to show that the limit is a solution to the Cauchy problem \eqref{eq:cauchy}. 

Many potentials with the property \eqref{eq:potential} have been proposed in the literature, so that, at the present stage, we have potentials which bound both the change of the strength and the change of the speed of the waves across interactions, both in the GNL/LD case and in the general, non-convex, strictly hyperbolic case. In particular: 
\begin{itemize}

\item in 1965 J. Glimm \cite{gli65} proposed a potential $Q^{\rm Glimm}$ which bounds the change of the \emph{strength} of the waves across interactions for systems whose characteristic fields are GNL/LD.

\item Later, in 1977, it was proved by T.P. Liu in \cite{TaiPing} that the same potential $Q^{\rm Glimm}$ bounds also the difference in \emph{speed} of the waves across an interaction, again for GNL/LD systems. 

\item Let us consider now systems whose flux $F$ is just strictly hyperbolic, but characteristic fields are not, in general, GNL/LD. A first potential $Q^{\rm Liu}$ which bounds the change of the \emph{strength} of the waves across an interaction was introduced by Liu in \cite{TaiPing2}, for a flux $F$ with finite number of inflection points.  In the general setting, where no convexity hypothesis on $F$ is assumed, a suitable potential $\Qcubic$ was found by Bianchini in 2003 \cite{Bia}.

\item More recently, in 2014-2015, in the three papers \cite{BiaMod1, BiaMod2, BiaMod3}, Bianchini and the author constructed a potential $Q^{\rm BM}$ which bounds the difference in \emph{speed} of the waves across an interaction, again for systems without any convexity assumption on $F$.

\end{itemize}

\subsection{A ``forward in time'' interaction potential}

A peculiar feature of the potential $\Qbm$ (which differentiates it from $Q^{\rm Glimm}$, $Q^{\rm Liu}$ and $\Qcubic$) is that its value $\Qbm(\bar t)$ at time $\bar t$ involves not only the profile $u^\e(\bar t)$ of the approximate solution at time $\bar t$, but also all of its past and future history, $\{u^\e(t)\}_{t \in [0, \infty)}$, making the definition of $\Qbm$ pretty complicated. Moreover, the fact that $\Qbm$ depends also on the past profiles $\{u^\e(t)\}_{t < \bar t}$ is counterintuitive, in the following sense. 

It is well known that the evolution generated by systems of conservation laws is unique forward in time but not backward in time. Therefore, it is natural to expect that an object defined ``at time $\bar t$'' depends on the present profile $u^\e(\bar t)$ and possibly on the future ones $\{u^\e(t)\}_{t > \bar t}$, but not on the past ones $\{u^\e(t)\}_{t < \bar t}$. Hence, a natural question (proposed by Bressan) is whether it is possible to define a uniformly bounded potential $\Upsilon$, which bounds the change of the speed of the waves at each interaction
\begin{equation}
\label{eq:cambio:vel}
\begin{split}
& \begin{split}
& |s(\bar t+, \alpha)||\sigma(\bar t-, \alpha) - \sigma(\bar t+, \alpha)| \\ 
& \qquad \qquad  + |s(\bar t+, \alpha')||\sigma(\bar t-, \alpha') - \sigma(\bar t+, \alpha')| 
\end{split}
\end{split}
\quad \leq \quad \Upsilon(\bar t - ) - \Upsilon(\bar t+),
\end{equation}
and whose value $\Upsilon(\bar t)$ depends only on the present and future profiles $\{u^\e(t)\}_{t \geq \bar t}$ and not on the past ones $\{u^\e(t)\}_{t < \bar t}$. 

A first answer  (proposed by Bressan)  to this question is to define $\Upsilon(\bar t)$ as the infimum of the potential $\Qbm(\bar t)$ over all past histories leading to the same profile $u^\e(\bar t)$. However, such definition is not explicit and such infimum may not be easy to characterize.

We thus ask whether it is possible to construct a potential $\Upsilon$ with a more explicit definition such that 
\begin{enumerate}
\item $\Upsilon(0) \leq C(F, \bar u)$;
\item $\Upsilon$ is decreasing and at each interaction \eqref{eq:cambio:vel} holds;
\item the value $\Upsilon(\bar t)$ at time $\bar t$ depends on the present and the future profiles of the approximate solution but not on the past ones.
\end{enumerate}

The answer to this question is positive and the aim of these notes is to present the construction of such  functional $\Upsilon$. 

Actually, we do not present the construction of $\Upsilon$ in the most general case, i.e. the case of an approximate solution to the system \eqref{eq:cauchy}, as we did in \cite{BiaMod3}, because this would lead to a huge paper full of technicalities. On the contrary, we present, in Section \ref{sec:core}, the detailed construction of $\Upsilon$ in the case of a wavefront tracking approximate solution to a scalar conservation law $u_t + F(u)_x = 0$, $u \in \R$, with a smooth flux $F: \R \to \R$ which, in general, does not satisfy any convexity assumption.

Such choice is not so restrictive for the following reason. The extension from the scalar case to the case of systems can be performed with the same techniques used in \cite{BiaMod2, BiaMod3}. A short summary of such techniques is contained in Section \ref{sec:extension}. On the contrary, the real novelty contained in this paper is that the new functional $\Upsilon$ not only satisfies Properties 1. and 2. above (this was already the case for $\Qbm$), but it also satisfies Property 3. However, it is not necessary to analyze the system case, in order to appreciate the features of $\Upsilon$ which lead to Property 3. On the contrary, all the novelties can be enjoyed through the analysis of a scalar wavefront tracking solution, with the advantage that, in this case, technical details can be mostly avoided.  


\subsection{A remark}

Differently from the potential $\Qbm$, the potentials $Q^{\rm Glimm}$, $Q^{\rm Liu}$, $\Qcubic$ and the potential $\Upsilon$ introduced in this paper have the property that their value at time $\bar t$ does not depend on the past profiles of the solution $\{u^\e(t)\}_{t < \bar t}$, but only on the present and the future ones $\{u^\e(t)\}_{t \geq \bar t}$. To be more precise, for the potential $Q^{\rm Glimm}$, $Q^{\rm Liu}$, $\Qcubic$, their value at time $\bar t$ depends \emph{only} on the present profile $u^\e(\bar t)$, while for the potential $\Upsilon$, its value  at time $\bar t$ truly depends both on the present and the future profiles $\{u^\e(t)\}_{t \geq \bar t}$. This distinction is meaningful for the following reason. 

All the interaction potentials introduced so far are defined on approximate solutions $u^\e$ constructed either through the wavefront tracking algorithm or through the Glimm scheme. If such potentials were defined on the exact solution $u$, then the distinction between
\begin{center}
{\bf S1}: ``the value of the potential at time $\bar t$ depends \\ on the present and the future profiles''
\end{center}
and
\begin{center}
{\bf S2}: ``the value of the potential at time $\bar t$ depends \\ only on the present profile''
\end{center}
would be meaningless, since, for the exact solution $u$, the profile $u(\bar t)$ at time $\bar t$ uniquely determines all the future profiles $u(t)$, $t \geq \bar t$, being the evolution unique forward in time. 

However, as remarked above, all the interaction potentials introduced so far are defined on approximate solutions $u^\e$, for which uniqueness forward in time is not true anymore. 

Indeed, assume that $u^\e$ is constructed by means of the Glimm scheme. Then, the profile $u^\e(\bar t)$ at time $\bar t$ uniquely determines the future evolution, only once the sampling sequence $\{\vartheta_i\}_{i \in \N}$ used in restarting procedure of the Glimm scheme is fixed. For the same datum $u^\e(\bar t)$ at time $\bar t$, different choices of the sampling sequence generate different evolutions at $t > \bar t$. 

Similarly, if $u^\e$ is constructed through the wavefront tracking algorithm, the profile $u^\e(\bar t)$ at time $\bar t$ uniquely determines the future evolution, only once we fix the threshold parameter $\rho>0$ which decides whether an accurate or a simplified Riemann solver has to be used. For the same datum $u^\e(\bar t)$ at time $\bar t$, different choices of the threshold $\rho$ generate different evolutions. 

Therefore, for approximate solutions, the two sentences {\bf S1} and {\bf S2} above are not equivalent. 


Remark \ref{rem:splitting} will clarify why in the definition of $\Upsilon$ we need to take into account not only the present profile of the approximate solution, but also the future ones.

\section{The quadratic potential for a wavefront tracking scalar solution}
\label{sec:core}

In this section we state and prove the main result of the paper, namely Theorem \ref{thm:main}, for an approximate solution to the scalar conservation law
\begin{equation}
\label{eq:cauchy:scalar}
\begin{cases}
u_t + F(u)_x = 0, \\
u(0, x) = \bar u(x) 
\end{cases}
\quad F: \R \to \R \text{ smooth}
\end{equation}
constructed through the wavefront tracking algorithm. First of all, in Section \ref{sec:main:thm}, we recall how the wavefront tracking algorithm works, we give a precise definition of \emph{change of the speed of the waves} across an interaction and we state Theorem \ref{thm:main}. Then in Sections \ref{sec:wavetracing}, \ref{sec:def:q}, \ref{sec:cancellation}, \ref{sec:interaction} we prove Theorem \ref{thm:main}.

\subsection{The main theorem}
\label{sec:main:thm}

First of all, we briefly recall how an entropic wavefront tracking solution is constructed, mainly with the aim to set the notations used later on.

For any $\e > 0$ let $F_\e$ be the piecewise affine interpolation of $F$ with grid size $\e$; let $\bar u^\e$ be an approximation of the initial datum $\bar u$ (in the sense that $\bar u^\e \rightarrow \bar u$ in $L^1$, as $\e \rightarrow 0$) of the Cauchy problem \eqref{eq:cauchy:scalar}, such that $\bar u^\e$ has compact support, it takes values in the discrete set $\Z\e$, and 
\begin{equation}
\label{bd_su_dato_iniziale}
\TV(\bar u^\e) \leq \TV(\bar u).
\end{equation}
Let $\xi_1 < \dots < \xi_K$ be the points where $\bar u^\e$ has a jump. 

At each $\xi_k$, consider the left and the right limits $\bar u^\e (\xi_k -), \bar u^\e(\xi_k+) \in \Z\e$. Solving the corresponding Riemann problems with flux function $F_\e$, we thus obtain a local entropic solution $u^\e(t,x)$, defined for $t$ sufficiently small, piecewise constant for each fixed time, and still taking values in $\Z\e$. 

From each $\xi_k$ some \emph{wavefronts} supported on discontinuity lines of $u^\e$ emerge. We prefer from now on to use the word \emph{wavefronts} instead of \emph{waves} to denote discontinuity lines, because the word \emph{wave} will have a slightly different meaning (see Section \ref{sec:wavetracing}). When two (or more) wavefronts meet (we will refer to this situation as an \emph{interaction}), we can again solve the new Riemann problem generated by the interaction, according to the above procedure, with flux $F_\e$, since the values of $u^\e(t,\cdot)$ always remain within the set $\Z\e$. 

The solution is then prolonged up to a time $t_2 > t_1$ where other wavefronts meet, and so on. One can prove \cite{Bre} that the total number of interaction points is finite, and hence the solution can be prolonged for all $t \geq 0$, thus defining a piecewise constant approximate solution $u^\e = u^\e(t,x)$, with values in the set $\Z\e$. We also assume that $u^\e(t, \cdot)$ is right continuous. 

Let $\{(t_j,x_j)\}$, $j \in \{1,2,\dots,J\}$, be the points in the $(t,x)$ plane where an interaction between two (or more) wavefronts occurs in the approximate solution $u^\e$. We also assume, without loss of generality, that $t_j < t_{j+1}$ and for every $j$ exactly two wavefronts meet in $(t_j, x_j)$. 
We also set $t_0 := 0$.

At each interaction $(t_j, x_j)$ we define the change of the speed $\Delta \sigma(t_j, x_j)$ of the wavefronts across the interaction as follows. Assume that the left wavefront carries the jump $(a,b)$, while the right wavefront carries the jump $(b,c)$, as in Figure \ref{fig:change:speed}. We have to distinguish between interaction between wavefronts having the same sign (so called \emph{same-sign interactions}, Figure \ref{fig:ssint}) and interaction between wavefronts having opposite sign (so called \emph{cancellations}, Figure \ref{fig:canc}).

\begin{figure}
\centering

\subfloat[][Same-sign interaction \label{fig:ssint}]
{
\begin{tikzpicture}[scale=1.2]
\draw[->] (0,0) to (0,2.2);
\node[above] at (0,2.2) {\footnotesize $t$};
\node at (3,2) {\footnotesize $u(t,x)$};
\node at (3,1.5) {\footnotesize $a < b< c$};
\draw (1,0) to (1.5,1.5);
\draw (3,0) to (3/2,1.5);
\draw (3/2,1.5) to (1.5,2.2);

\draw[dotted] (1.5, 1.5) to (0, 1.5);
\node[left] at (0,1.5) {\footnotesize $t_j$};

\node at (0.8, 0.6) {\footnotesize $a$};
\node at (1.8, 0.6) {\footnotesize $b$};
\node at (2.7, 0.6) {\footnotesize $c$};
\end{tikzpicture}
}
\qquad
\subfloat[][Cancellation \label{fig:canc}]
{
\begin{tikzpicture}[scale=1.2]



\draw[->] (0,0) to (0,2.2);
\node[above] at (0,2.2) {\footnotesize $t$};
\node at (3.5,2) {\footnotesize $u(t,x)$};
\node at (3.5,1.5) {\footnotesize $a<c<b$};
\draw (1,0) to (1.5,1.5);
\draw[dashed] (3,0) to (3/2,1.5);
\draw (1.5, 1.5) to (2,2);
\draw (1.5, 1.5) to (2.2,2);
\draw (1.5, 1.5) to (2.4,2);
\draw (1.5, 1.5) to (2.6,2);
\draw (1.5, 1.5) to (2.8,2);
\draw[dotted] (1.5, 1.5) to (0, 1.5);
\node[left] at (0, 1.5) {\footnotesize $t_j$};

\node at (0.8, 0.5) {\footnotesize $a$};
\node at (1.9, 0.5) {\footnotesize $b$};
\node at (3, 0.5) {\footnotesize $c$};

\end{tikzpicture}
}
\caption{Different types of interaction}
\label{fig:change:speed}
\end{figure}
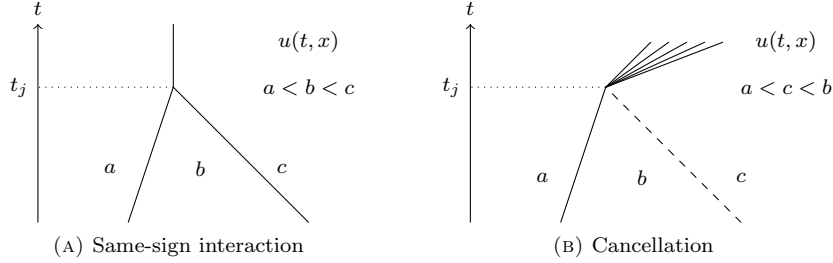

For same-sign interactions, assume that $a < b < c$ (in the case  $a > b > c$ a similar definition holds). We set
\begin{equation}
\label{eq:cambio:signa:int}
\begin{split}
\Delta \sigma(t_j, x_j) := 
\int_{a}^{b} \Big| & \frac{d}{d\tau} \conv_{[a, b]} F_\e(\tau) - \frac{d}{d\tau} \conv_{[a, c]} F_\e(\tau) \Big| d\tau \\
& + \int_{b}^{c} \Big| \frac{d}{d\tau} \conv_{[b, c]} F_\e(\tau) - \frac{d}{d\tau} \conv_{[a, c]} F_\e(\tau) \Big| d\tau.
\end{split}
\end{equation}

For cancellations, assume that $a < c < b$ (in the other cases similar definitions hold). We set
\begin{equation}
\label{eq:cambio:signa:canc}
\Delta \sigma(t_j, x_j) := 
\int_{a}^{c} \Big| \frac{d}{d\tau} \conv_{[a, c]} F_\e(\tau) - \frac{d}{d\tau} \conv_{[a, b]} F_\e(\tau) \Big| d\tau.
\end{equation}

Here $\conv_I g$ denotes the \emph{convex envelope} of a function $g$ on the interval $I$: $\conv_I g(\tau) := \sup \big\{ h(\tau) \ \big| \ h \leq g, \ h \text{ convex} \big\}$. A similar definition holds for the concave envelope $\conc_I g$. 

For a discussion about why \eqref{eq:cambio:signa:int}-\eqref{eq:cambio:signa:canc} are suitable definitions of the change of the speed of the wavefronts across the interaction, we refer to the Introductions to the papers \cite{BiaMod1, BiaMod2, BiaMod3}.

We can now state the main result of these notes. We stress once again that this result holds for conservation laws with a flux $F$ which satisfies no convexity assumption.

\begin{thm}
\label{thm:main}
There exists a functional $\Upsilon = \Upsilon(t)$
\begin{itemize}
\item explicitly defined by \eqref{eq:upsilon}, \eqref{eq:q};
\item whose value $\Upsilon(\bar t)$ at any time $\bar t$ depends only on the present and future profiles of the solution $u^\e(t)$, $t \geq \bar t$;
\end{itemize}
with the following properties:
\begin{enumerate}
\item $\Upsilon(0) \leq \|F''\|_{L^\infty}  \TV(\bar u)^2$;
\item $\Upsilon$ is decreasing and at each \emph{interaction} time $t_j$, it holds
\begin{equation}
\label{eq:stima}
\Delta \sigma(t_j, x_j) \leq \Upsilon(t_j - ) - \Upsilon(t_j +). 
\end{equation}
\end{enumerate}
\end{thm}



The hard part in the proof of Theorem \ref{thm:main} is to construct a potential which satisfies \eqref{eq:stima} in the same-sign interaction points, i.e. to prove the following proposition.

\begin{prop}
\label{prop:main}
There exists a functional $\fQ = \fQ(t)$
\begin{itemize}
\item explicitly defined by \eqref{eq:q};
\item whose value $\fQ(\bar t)$ at any time $\bar t$ depends only on the present and future profiles of the solution $u^\e(t)$, $t \geq \bar t$;
\end{itemize}
with the following properties:
\begin{enumerate}
\item  $\fQ(0) \leq \|F''\|_{L^\infty}  \TV(\bar u)^2$;
\item $\fQ$ is decreasing and at each \emph{same-sign interaction} time $t_j$, it holds
\begin{equation}
\label{eq:bound:fq}
\Delta \sigma(t_j, x_j) \leq \fQ(t_j - ) - \fQ(t_j +). 
\end{equation}
\end{enumerate}
\end{prop}

Assuming Proposition \ref{prop:main}, the proof of Theorem \ref{thm:main} follows easily.

\begin{proof}[Proof of Theorem \ref{thm:main} (assuming Proposition \ref{prop:main})]
It $(t_j, x_j)$ is a cancellation point, assuming as above that $a < c < b$, then it is non difficult to show that 
\begin{equation*}
\begin{split}
\Delta \sigma (t_j, x_j) & \leq \|F''\|_{L^\infty} |c - a| |c - b| \\
& \leq  \ \|F''\|_{L^\infty} \TV(\bar u) \Big[ \TV( u^\e(t-)) - \TV(u^\e(t+)) \Big].
\end{split}
\end{equation*}
Therefore the functional 
\begin{equation*}
t \mapsto  \|F''\|_{L^\infty} \TV(\bar u) \TV( u^\e(t, \cdot); \R)
\end{equation*}
bounds the change of speed across cancellations. Setting
\begin{equation}
\label{eq:upsilon}
\Upsilon(t) :=  \|F''\|_{L^\infty} \TV(\bar u) \TV( u^\e(t, \cdot); \R) + \fQ(t),
\end{equation}
we conclude the proof of Theorem \ref{thm:main}.
\end{proof}
 
%
%
%
%
%
%

The proof of Proposition \ref{prop:main}, i.e. the construction of the potential $\fQ$, is the core of this paper and it is performed in the next sections. In Section \ref{sec:wavetracing} we recall the fundamental notion of \emph{wave tracing}, in Section \ref{sec:def:q} we define the interaction potential $\fQ$, in Section \ref{sec:cancellation} we prove that $\fQ$ is decreasing across cancellations and finally in Section \ref{sec:interaction} we show that $\fQ$ is decreasing also across same-sign interactions and it satisfies \eqref{eq:bound:fq}.

\subsection{The wave tracing}
\label{sec:wavetracing}

The first step to prove Theorem \ref{thm:main} is the introduction of a so called \emph{wave tracing}. In his work \cite{TaiPing}, Liu showed that it is possible to decompose each jump in the approximate initial datum $u^\e(0, \cdot)$ into small (or even infinitesimal) pieces, which we call \emph{waves}, such that the position and the speed of each wave can be traced from time $0$ to the time when the wave is possibly canceled. Such \emph{wave tracing} is constructed as follows.\footnote{To be precise, the wave tracing we present here is not exactly Liu's original one. Rather, it is a modified version (suitable for scalar wavefront tracking solutions, but extendable also to the Glimm scheme and to systems), in which each wave has constant infinitesimal strength. See, for instance, \cite[Section 4]{BiaMod3}.} 

\subsubsection{Construction of the wave tracing}
\label{sss:construction:wt}

First of all set $$\bar U^\e(x) := \TV(\bar u^\e; (-\infty, x]).$$ We define the \emph{set of waves} as $\W:= (0, \TV(\bar u^\e)] = (0, \bar U^\e(+\infty)]$. 

For each wave $w \in \W$, we want now to define  its \emph{sign} $\mathcal S(w)$, its \emph{state} $\mathtt u(w)$, its \emph{cancellation time} $\tcanc(w) \in (0, \infty]$ and its \emph{position} $\mathtt X(t, w)$, $t \in [0, \tcanc(w))$. All these objects clearly depend on $\e$, but we omit to write such dependence explicitly.  

The \emph{position of the wave $w \in \W$ at time $0$} is defined as the unique point $\mathtt X(0,w)$ such that 
\begin{equation*}
\bar U^\e (\mathtt X(0,w)-) < w \leq  \bar U^\e (\mathtt X(0,w)+).
\end{equation*}
The \emph{sign} of $w$ is defined as
\begin{equation*}
\mathcal S(w) := 
\begin{cases}
1 &  \text{if $u^\e$ has a positive jump at $\mathtt X(0, w)$}, \\
-1 & \text{if $u^\e$ has a negative jump at $\mathtt X(0, w)$}.
\end{cases}
\end{equation*}
We then define the \emph{state} of $w$ as
\begin{equation*}
\mathtt u(w) := \int_0^w \mathcal S(w') dw'.
\end{equation*}

The definition of the sets 
\begin{equation*}
\W(t) := \{s \in \W \ \big| \ \tcanc(s) > t \big\}, \quad t \geq 0,
\end{equation*}
(which is an implicit definition of the map $\tcanc: \W \to (0, +\infty]$) and the definition of the position $\mathtt X$ for times $t > 0$ is given by a recursive procedure, partitioning the time interval $[0, +\infty)$ in the following way
\begin{equation}
\label{eq:recursion}
[0, +\infty) = \{0\} \cup (0, t_1] \cup \dots \cup (t_j, t_{j+1}] \cup \dots \cup (t_{J-1}, t_J] \cup (t_J, +\infty),
\end{equation}
assuming that $\W(t)$ and $\mathtt X(t, \cdot)$ are defined for any time $t \leq t_j$ and defining $\W(t)$ and $\mathtt X(t, \cdot)$ on the time interval $(t_j, t_{j+1}]$, with the property that $\mathcal S$ is constant on
\begin{equation*}
\W(t,x) := \mathtt X(t)^{-1}(x)
\end{equation*}
and the map $\W(t) \ni w \mapsto \mathtt X(t,w)$ is non-decreasing. 

We have already defined $\mathtt X(0, \cdot)$ at time $t=0$; moreover, since the codomain of $\tcanc$ is $(0, \infty]$, it holds also $\W(0) = \W$. 

Assume now we have already defined $\W(t)$ and $\mathtt X(t,\cdot)$ for every $t \leq t_j$ and let us define it for $t \in (t_j, t_{j+1}]$ (or $t \in (t_J, +\infty)$). For any $t \leq t_{j}$ set 
\begin{equation}
\label{eq:sigma}
\sigma(t,w) := 
\begin{cases}
\frac{d}{du}\conv_{\mathtt u ( \W(t, \mathtt X(t,w)))} F_\e (\mathtt u(w) ) & \text{if } \mathcal S(w) = +1, \\
\frac{d}{du}\conc_{\mathtt u ( \W(t,\mathtt X(t,w)))} F_\e (\mathtt u(w) ) & \text{if } \mathcal S(w) = -1.
\end{cases}
\end{equation}
Using the fact that all the waves in $\W(t,x)$ have the same sign, it is not difficult to show that $\mathtt u ( \W(t, \mathtt X(t,w)))$ is an interval in $\R$ and thus definition \eqref{eq:sigma} is well posed. 
Observe also that $\sigma(t,w)$ is the Rankine-Hugoniot speed of the wavefront containing $w$ at time $t$.
 
\noindent We can now define $\W(t)$ and $\mathtt X(t,\cdot)$ for $t \in (t_j, t_{j+1}]$, as follows. Introduce first the notation
\begin{equation*}
I(a,b) :=
\begin{cases}
(a, b] & \text{if } a < b, \\
(b, a] & \text{if } b < a.
\end{cases}
\end{equation*}
\begin{enumerate}
\item \label{pt:t<tj} For $t < t_{j+1}$ (or $t_J < t < +\infty$) and for any $w \in \W(t)$, set
\begin{equation}
\label{eq:position}
\W(t) = \W(t_j), \qquad \mathtt X(t,w) := \mathtt X(t_j, w) + \sigma(t_j, w)(t-t_j).
\end{equation}

\item \label{pt:t=tj} For $t=t_{j+1}$:

    \begin{enumerate} 
    \item if $\mathtt X(t_j,s) +\sigma(t_j,s)(t_{j+1} - t_j) \neq x_{j+1}$, set
    \begin{equation*}
    \mathtt X(t_{j+1},w) := \mathtt X(t_j,w) +\sigma(t_j,w)(t_{j+1} - t_j)
    \end{equation*}

    \item if $\mathtt X(t_j,w) +\sigma(t_j,w)(t_{j+1} - t_j) =  x_{j+1}$, we distinguish two further cases:

        \begin{enumerate}
        
        \item if $\mathtt u(w) \in I ( u^\e(t_{j+1},x_{j+1}-), u^\e(t_{j+1},x_{j+1}))$
        (i.e. $w$ ``survives'' the possible cancellation in $(t_{j+1},x_{j+1})$)
        we define
        \[
        \mathtt X(t_{j+1},s) :=  x_{j+1}.
        \]
        
        \item otherwise, if $w \in \W(t_j)$, $\mathtt X(t_j,w) +\sigma(t_j,w)(t_{j+1} - t_j) =  x_{j+1}$, but $\mathtt u(w) \notin I ( u^\e(t_{j+1},x_{j+1}-), u^\e(t_{j+1},x_{j+1}))$, (i.e. $w$ is canceled by the possible cancellation in $(t_{j+1},x_{j+1})$), we assign $\tcanc(s) := t_{j+1}$, thus defining $\W(t_{j+1})$. 

        \end{enumerate}
    \end{enumerate}
\end{enumerate}

Notice that, as far as $t < \tcanc(w)$, it holds $\sigma(t,w) = \partial_t \mathtt X(t,w)$.

\subsubsection{Further objects related to the wave tracing}

We conclude this section introducing some further notions related to the wave tracing, which will be frequently used in the next sections.

\begin{lem}
\label{lem:scambio:onde:stati:easy}
The following holds.
\begin{enumerate}
\item Any set $\W(t,x)$ is a finite union (possibly empty) of pairwise disjoint intervals of real numbers. 
\item The restriction $\mathtt u|_{\W(t,x)}$ is affine with slope equal to $\mathcal S(\W(t,x))$.
\item It holds $\mathtt u(\W(t,x)) = I(u(t,x-), u(t,x+))$.
\end{enumerate}
\end{lem}
\begin{proof}
The first claim follows from the observation that, by construction, $\W(t)$ is a finite union of pairwise disjoint intervals of real numbers and from the monotonicity of $\W(t) \ni w \mapsto \mathtt X(t,w)$. 

The second claim follows from the fact that $\mathcal S$ is constant on $\W(t,x)$. 

The third claim is proved by induction, as in \eqref{eq:recursion}. At time $t=0$, the claim follows from the definition of wave tracing. Assume now the claim is true on the time interval $(t_{j-1}, t_j]$ and let us prove it on $(t_j, t_{j+1}]$. 
We distinguish two cases.

\begin{enumerate} 
\item If $t \in (t_j, t_{j+1})$, then the claim follows from the construction described at Point \ref{pt:t<tj}, page \pageref{pt:t<tj}, and, in particular, from 
   \eqref{eq:sigma}-\eqref{eq:position}. 
\item If $t = t_{j+1}$, then $\W(t,x) = \W(t_{j+1}, x_{j+1})$ and the claim follows from the construction described at Point \ref{pt:t=tj}, page \pageref{pt:t=tj}. \qedhere
\end{enumerate}
\end{proof}

\begin{defn}
\label{W_iow_def}
Let $\mathcal{I} \subseteq \W(t)$. We say that $\mathcal{I}$ is an \emph{interval of waves at time $t$} if all the waves in $\mathcal I$ have the same sign and, moreover, for any given $w,w' \in \mathcal{I}$, with $w \leq w'$, and for any $z \in \W(t)$
\[
w \leq z \leq w' \Longrightarrow z \in \mathcal{I}.
\]
\end{defn}

\begin{lem}
\label{lem:scambio:onde:stati}
Any interval of waves $\mathcal I$ is a finite union (possibly empty) of pairwise disjoint intervals of real numbers. Moreover the restriction $\mathtt u|_{\mathcal I}$ is affine with slope equal to $\mathcal S(\mathcal I)$ on each of such intervals and the image $\mathtt u(\mathcal I)$ is an interval in $\R$.
\end{lem}
\begin{proof}
If $\mathcal I$ is an interval of waves at time $t$, then, by Definition \ref{W_iow_def} and by the fact that $\W(t) \ni w \mapsto \mathtt X(t,w)$ is non-decreasing, there exist points $x_0 < x_1 < \cdots < x_N$ and two waves $w_0 \in \W(t, x_0)$, $w_N \in \W(t, x_N)$ such that
\begin{equation*}
\begin{split}
& \mathcal A_0 = (w_0, + \infty) \cap \W(t,x_0) \text{ or } \mathcal A_0 = [w_0, + \infty) \cap \W(t,x_0), \\
& \mathcal A_N = (-\infty, w_N) \cap \W(t,x_N) \text{ or } \mathcal A_N = (-\infty, w_N] \cap \W(t,x_0) \\
\end{split}
\end{equation*}
and
\begin{equation*}
\mathcal I  = \mathcal A_0 \cup \bigcup_{n=1}^{N-1} \W(t,x_n) \cup \mathcal A_N.
\end{equation*}
The conclusion now follows applying Lemma \ref{lem:scambio:onde:stati:easy}. 
\end{proof}



\begin{rem}
\label{rem:scambio:onde:stati}
Thanks to the previous lemma, with a slight abuse of notation, we will often identify $\mathcal I$ with its image $\mathtt u(\mathcal I)$. 
\end{rem}


With this in mind, it makes sense to define the \emph{entropic speed} given to a wave $w \in \mathcal I$ by the interval $\mathcal I$ as
\begin{equation*}
\sigmaent(\mathcal I, w) :=
\begin{cases}
\frac{d}{du} \conv_{\mathtt u(\mathcal I)} F_\e (\mathtt u(w)) & \text{if } \mathcal S(\mathcal I) = +1, \\
\frac{d}{du} \conc_{\mathtt u(\mathcal I)} F_\e (\mathtt u(w)) & \text{if } \mathcal S(\mathcal I) = -1, \\
\end{cases}
\end{equation*}
and the \emph{Rankine-Hugoniot speed} of the interval $\mathcal I$ as
\begin{equation*}
\sigmarh(\mathcal I) := \frac{F_\e(\sup \mathtt u(\mathcal I)) - F_\e(\inf \mathtt u(\mathcal I))}{\sup \mathtt u(\mathcal I) - \inf \mathtt u(\mathcal I)}.
\end{equation*}

\begin{defn}
We say that two waves $w,w' \in \W(\bar t)$ interact at time $\bar t$ if $\mathtt X(\bar t,w) = \mathtt X(\bar t,w')$. We say that they will interact after time $\bar t$ if there is a time $t > \bar t$ such that $w,w'$ interact at time $t$. If $w,w' \in \W(\bar t)$ will interact after time $t$, we define
\begin{equation*}
\tint(\bar t,w,w') := \min \big \{t > \bar t \ \big| \ \mathtt X(t,w) = \mathtt X(t,w') \big\}
\end{equation*}
and
\begin{equation*}
\xint(\bar t,w,w') := \mathtt X(\tint(\bar t,w,w'), w) = \mathtt X(\tint(\bar t,w,w'), w').
\end{equation*}
\end{defn}

\subsection{Definition of the interaction potential $\fQ$}
\label{sec:def:q}

We are finally in the position to define the interaction potential as
\begin{equation}
\label{eq:q}
\fQ(\bar t) := \iint_{\substack{w,w' \in \W(\bar t) \\  w < w'}} \mathfrak q(\bar t,w,w') dw dw',
\end{equation}
where the \emph{weight} $\mathfrak q(\bar t,w,w')$ of the pair $(w,w')$ at time $\bar t$ is defined as follows. 

Fix $w,w' \in \W(\bar t)$, $w<w'$. First of all let us rule out the following two cases:
\begin{itemize}
\item  if $\mathcal S$ is not constant on $[w,w'] \cap \W(\bar t)$, then set $\mathfrak q(\bar t,w,w') := \|F''\|_{L^\infty}$;
\item  if $\mathcal S$ is constant on $[w,w'] \cap \W(\bar t)$, but either $w,w'$ have the same position at time $\bar t$ or they will never interact after time $\bar t$, then set $\mathfrak q(\bar t,w,w') := 0$.
\end{itemize}

We have still to define $\mathfrak q(\bar t,w,w')$ for pair of waves $w,w' \in \W(\bar t)$ such that all the waves in $[w,w'] \cap \W(\bar t)$ have the same sign, $\mathtt X(\bar t,w) < \mathtt X(\bar t,w')$ and $w,w'$ will interact after time $\bar t$.

Assume for definiteness that $w,w'$ and all the waves in $[w,w'] \cap \W(\bar t)$ are positive, the negative case being similar. Define
\begin{equation}
\label{eq:Wtww}
\W(\bar t,w,w') := \W\big(\tint(\bar t,w,w'), \xint(\bar t,w,w') \big) 
\end{equation}
the set of all waves which will have the same position of $w,w'$ at the first time when $w,w'$ will interact after time $\bar t$. Define also
\begin{equation}
\label{E:jj}
\begin{split}
& \jleft(\bar t,w,w') := \mathtt X^{-1}(\bar t) \big( \mathtt X(\bar t,w) \big) \cap \W(\bar t,w,w'), \\
& \jright(\bar t,w,w') := 
\mathtt X^{-1}(\bar t) \big(\mathtt X(\bar t,w') \big) \cap \W(\bar t,w,w').
\end{split}
\end{equation}
Roughly speaking, the set $\jleft(\bar t,w,w')$ (resp. $\jright(\bar t,w,w')$) consists of all the waves which have the same position of $w$ (resp. $w'$) at time $\bar t$ and which will have the same position of $w$ and $w'$ at the time in which $w$ and $w'$ will interact. 

\begin{lem}
The sets $\jleft(\bar t,w,w'), \jright(\bar t,w,w')$ are intervals of waves at time $\bar t$. 
\end{lem}
\begin{proof}
We prove that $\jleft(\bar t,w,w')$ is an interval of waves at time $\bar t$; a similar proof holds for $\jright(\bar t,w,w')$. Let $z , z',z'' \in \W(\bar t)$ be given, $z < z' < z''$, and assume that $z,z'' \in \jleft(\bar t,w,w')$. We want to prove that $z' \in \jleft(\bar t,w,w')$. Being $\W(t) \ni w \mapsto \mathtt X(t, w)$ non-decreasing for each fixed time $t$, we have
\begin{equation*}
\mathtt X(\bar t,w) = \mathtt X(\bar t,z) \leq \mathtt X(\bar t, z') \leq  \mathtt X(\bar t,z'') = \mathtt X(\bar t,w)
\end{equation*}
and thus $\mathtt X(\bar t,z') = \mathtt X(\bar t,w)$. A entirely similar argument at the interaction time $\tint(\bar t,w,w')$ implies that $\mathtt X(\tint(\bar t,w,w'), z') = \mathtt X(\tint(\bar t,w,w'), w)$ and thus $z' \in \jleft(\bar t,w,w')$. 
\end{proof}

We can now define the weight:
\begin{equation}
\label{E:weight}
\mathfrak q(\bar t,w,w'):=
\frac{\pi(\bar t,w,w')}{d(\bar t,w,w')},
\end{equation} 
where
\begin{equation}
\label{eq:pi}
\pi(\bar t,w,w') := \Big[\sigmaent \big(w, \jleft(\bar t,w,w')   \big) - \sigmaent \big(w', \jright(\bar t,w,w')  \big)\Big]^+ 
\end{equation}
and
\begin{equation}
\label{eq:di}
d(\bar t,w,w') := \mathcal L^1 \Big( \W(\bar t,w,w')\Big).
\end{equation}
By definition of $\jleft(\bar t,w,w')$, $\jright(\bar t,w,w')$, it is not difficult to see that $\mathfrak q(\bar t,w,w') \leq \|F''\|_\infty$, and thus
\begin{equation}
\label{eq:unif:bd}
\fQ(\bar t) \leq \|F''\|_{L^\infty} \TV(u^\e(t))^2.
\end{equation}

\begin{rem}
Notice that, in general, the speed $\sigmaent (w, \jleft(\bar t,w,w'))$ (resp. $\sigmaent (w', \jright(\bar t,w,w') )$) does not coincide with the Rankine-Hugoniot speed of the jump located in $\mathtt X(\bar t,w)$ (resp. $\mathtt X(\bar t,w')$). See also the Example at the end of this Section. 
\end{rem}

\begin{rem}
The functional $\fQ$ defined by \eqref{eq:q} has the property that its value at time $\bar t$ depends only on the profiles $u^\e(t)$, $t \geq \bar t$ and not on the past profiles. Moreover it satisfies the bound \eqref{eq:unif:bd} and it is constant on each interval $(t_j, t_{j+1})$. Therefore, in order to complete the proof of Proposition \ref{prop:main} it is enough to prove that 
\begin{itemize}
\item it decreases across cancellations (Section \ref{sec:cancellation});
\item it decreases across same-sign interactions and the decrease across same-sign interactions satisfies $\Delta \sigma(t_j, x_j) \leq \fQ(t_j-) - \fQ(t_j+)$ (Section \ref{sec:interaction}).
\end{itemize}
\end{rem}


The following lemma presents the fundamental property of the intervals $\jleft(\bar t,w,w'), \jright(\bar t,w,w')$ defined above. 

\begin{lem}
\label{L:fund:prop}
Let $w,w',w'' \in \W(\bar t)$ be three fixed waves, $w \leq w' \leq w''$. Then
\begin{equation*}
\jright(\bar t,w,w') = \jright(\bar t,w,w'') 
\quad \Longrightarrow \quad 
\jleft(\bar t,w,w') = \jleft(\bar t,w,w'').
\end{equation*}
\end{lem}
\begin{proof}
Observe that 
\begin{equation*}
w'' \in \jright(\bar t, w,w'') = \jright(\bar t, w, w') = \mathtt X(\bar t)^{-1} \big(\mathtt X(t,w') \big) \cap \W(\bar t,w,w').
\end{equation*}
In particular $w'' \in \W(\bar t,w,w')$. Hence 
\begin{equation*}
\mathtt X \big(\tint(\bar t,w,w'), w\big) = \mathtt X \big(\tint(\bar t,w,w'), w'\big) = \mathtt X\big(\tint(\bar t,w,w'),w''\big),
\end{equation*}
and thus $\tint(\bar t, w,w'') \leq \tint(\bar t,w,w')$. On the other side, being $\W(t) \ni w \mapsto \mathtt X(t, w)$ non-decreasing for each fixed $t$, we must have also $\tint(\bar t,w,w'') \geq \tint(\bar t,w,w')$ and thus $\tint(\bar t,w,w') = \tint(\bar t ,w,w')$. As a consequence, 
\begin{equation*}
\begin{split}
\W(\bar t,w,w') 
& = \W\big(\tint(\bar t,w,w'), \xint(\bar t,w,w') \big) \\
& = \W\big(\tint(\bar t,w,w''), \xint(\bar t,w,w'') \big) \\
& = \W(\bar t,w,w''), 
\end{split}
\end{equation*}
from which the conclusion follows. 
%
\end{proof}

\begin{ex}
\label{example}
We give an explicit example of computation of $\mathfrak Q$. Let us assume that $u^\e|_{[t_0, +\infty) \times \R}$ is as in Figure \ref{fig:example}.
 
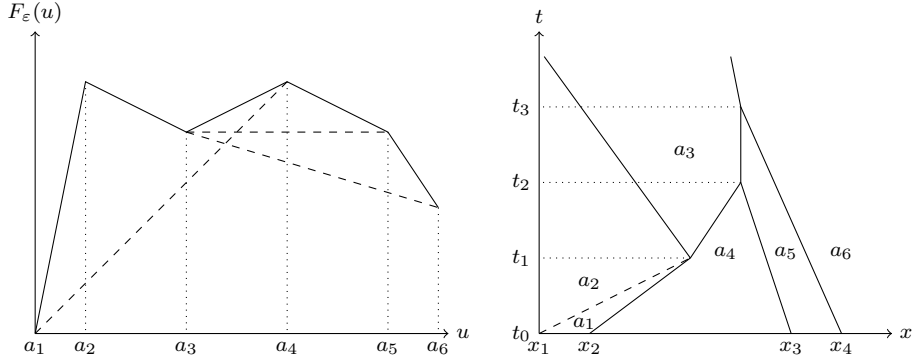
\begin{figure}
\centering
\begin{tikzpicture}[scale=0.67]
\footnotesize

\coordinate (a1) at (0-0.5,0); \coordinate (b1) at (0-0.5,0);
\coordinate (a2) at (1-0.5,5); \coordinate (b2) at (1-0.5,0);
\coordinate (a3) at (3-0.5,4); \coordinate (b3) at (3-0.5,0);
\coordinate (a4) at (5-0.5,5); \coordinate (b4) at (5-0.5,0);
\coordinate (a5) at (7-0.5,4); \coordinate (b5) at (7-0.5,0);
\coordinate (a6) at (8-0.5,2.5); \coordinate (b6) at (8-0.5,0);

\draw (a1) to (a2) to (a3) to (a4) to (a5) to (a6);
\draw[dashed] (a1) to (a4);
\draw[dashed] (a3) to (a5);
\draw[dashed] (a3) to (a6);

\draw[dotted] (a2) to (b2);
\draw[dotted] (a3) to (b3);
\draw[dotted] (a4) to (b4);
\draw[dotted] (a5) to (b5);
\draw[dotted] (a6) to (b6);

\node[below] at (b1) {$a_1$};
\node[below] at (b2) {$a_2$};
\node[below] at (b3) {$a_3$};
\node[below] at (b4) {$a_4$};
\node[below] at (b5) {$a_5$};
\node[below] at (b6) {$a_6$};

\draw [->] (0-0.5,0) to (0-0.5,6);
\draw [->] (0-0.5,0) to (8.2-0.5,0);

\node[right] at (8.2-0.5,0) {$u$};
\node[above] at (0-0.5,6) {$F_\e (u)$};

\draw [->] (9.5,0) to (16.5,0);
\draw[->] (9.5,0) to (9.5,6);
\node[right] at (16.5,0) {$x$};
\node[above] at (9.5,6) {$t$};

\coordinate (c1) at (9.5,0);
\coordinate (c2) at (10.5,0);
\coordinate (c3) at (12.5,1.5);
\coordinate (c4) at (9.6,5.5);
\coordinate (c5) at (13.5,3);
\coordinate (c6) at (13.5,4.5);
\coordinate (c7) at (13.3,5.5);
\coordinate (c8) at (14.5,0);
\coordinate (c9) at (15.5,0);
\coordinate (c10) at (9.5,1.5);
\coordinate (c11) at (9.5, 3);
\coordinate (c12) at (9.5, 4.5);

\draw[dashed] (c1) to (c3);
\draw (c2) to (c3);
\draw (c3) to (c4);
\draw (c3) to (c5) to (c6) to (c7);
\draw (c8) to (c5);
\draw (c9) to (c6);
\draw[dotted] (c3) to (c10);
\draw[dotted] (c5) to (c11);
\draw[dotted] (c6) to (c12);

\node[left] at (c10) {$t_1$};
\node[left] at (c11) {$t_2$};
\node[left] at (c12) {$t_3$};
\node[left] at (c1) {$t_0$};
\node[below] at (c1) {$x_1$};
\node[below] at (c2) {$x_2$};
\node[below] at (c8) {$x_3$};
\node[below] at (c9) {$x_4$};

\node at (10.5,1) {$a_2$};
\node at (10.4,0.2) {$a_1$};
\node at (13.2,1.6) {$a_4$};
\node at (14.4,1.6) {$a_5$};
\node at (15.5,1.6) {$a_6$};
\node at (12.4,3.6) {$a_3$};

\normalsize

\end{tikzpicture}
\caption{The situation described in the Example at page \pageref{example}. The figure on the left describes the flux $u \mapsto F_\e(u)$; the figure on the right describes the solution $u^\e(t,x)$ in the $(t,x)$-plane.}
\label{fig:example}
\end{figure}

By our definition of the wave tracing, the set of waves $\W(t_0)$ at time $t_0$ must consist of a finite union of intervals of total length $a_2 - a_1$ corresponding to the negative waves located at $x_1$ and a finite union of intervals of total length $a_6 - a_1$ corresponding to positive waves located in $x_2, x_3, x_4$. We can thus assume w.l.o.g. that 
\begin{equation*}
\W(t_0) = (a_1 - (a_2 - a_1), a_1] \cup (a_1, a_7]
\end{equation*}
The waves in $(a_1 - (a_2 - a_1), a_1]$ are the negative waves located in $x_1$;  the waves in $(a_1, a_4]$ are those located in $x_2$, the waves in $(a_4,a_5]$ are those located in $x_3$ and the waves in $(a_5, a_6]$ are those located in $x_4$. 

After the cancellation takes place at time $t_1$, the negative waves in $(a_1 - (a_2 - a_1), a_1]$ and the positive waves in $(a_1, a_2]$ are canceled, while those in $(a_2, a_4]$ are split into two subsets $(a_2, a_3] \cup (a_3, a_4]$ following distinct trajectories. 

Since all negative waves have the same position at time $t_0$, $\mathfrak q(t_0,w,w') = 0$ for each pair $(w,w')$ of negative waves. We thus need to compute $\mathfrak q(t_0,w,w')$ only for pairs of positive waves. We can assume that $w,w'$ have different position (otherwise $\mathfrak q(t_0,w,w')=0$, by definition). We have:

%


\begin{itemize}

\item if $w \in (a_1, a_3]$, then, for every $w' > a_3$, $w,w'$ will never interact after time $t_0$ and therefore $\mathfrak q(t_0,w,w') = 0$;

\item if $w \in (a_3, a_4]$ and $w' \in (a_4, a_5]$, then, $\tint(t_0,w,w') = t_2$ and thus, according to \eqref{eq:Wtww}-\eqref{E:jj}, 
\begin{equation*}
\W(t_0,w,w') = (a_3,a_5], \quad \jleft(t_0,w,w') = (a_3,a_4], \quad \jright(t_0,w,w') = (a_4,a_5];
\end{equation*}
therefore $d(t_0,w,w') = a_5 - a_3$ and
\begin{equation*}
\begin{split}
\pi(t_0,w,w') & :=  \Big[ \conv_{[a_3,a_4]} F_\e(w) - \conv_{[a_4,a_5]]} F_\e(w') \Big]^+ \\
& =   \bigg[ \frac{F_\e(a_4) - F_\e(a_3)}{a_4 - a_3} - \frac{F_\e(a_5) - F_\e(a_4)}{a_5 - a_4} \bigg]^+;
\end{split}
\end{equation*}

\item similarly, if $w \in (a_3,a_4]$ and $w' \in (a_5,a_6]$, then $\tint(t_0,w,w') = t_3$ and thus
\begin{equation*}
\W(t_0,w,w') = (a_3,a_6], \quad \jleft(t_0,w,w') = (a_3,a_4], \quad \jright(t_0,w,w') = (a_5,a_6];
\end{equation*}
therefore $d(t_0,w,w') = a_6 - a_3$ and
\begin{equation*}
\begin{split}
\pi(t_0,w,w') & := \Big[ \conv_{[a_3,a_4]} F_\e(w) - \conv_{[a_5,a_6]} F_\e(w') \Big]^+ \\
& = \bigg[ \frac{F_\e(a_4) - F_\e(a_3)}{a_4 - a_3} - \frac{F_\e(a_6) - F_\e(a_5)}{a_6 - a_5} \bigg]^+;
\end{split}
\end{equation*}

\item finally, if $w \in (a_4,a_5]$ and $w' \in (a_5,a_6]$, then $\tint(t_0,w,w') = t_3$ and thus
\begin{equation*}
\W(t_0,w,w') = (a_3,a_6], \quad \jleft(t_0,w,w') = (a_4,a_5], \quad \jright(t_0,w,w') = (a_5,a_6];
\end{equation*}
therefore $d(t_0,w,w') = a_6 - a_3$ and
\begin{equation*}
\begin{split}
\pi(t_0,w,w') & := \Big[ \conv_{[a_4,a_5]} F_\e(w) - \conv_{[a_5,a_6]} F_\e(w) \Big]^+ \\
& = \bigg[ \frac{F_\e(a_5) - F_\e(a_4)}{a_5 - a_4} - \frac{F_\e(a_6) - F_\e(a_5)}{a_6 - a_5} \bigg]^+. 
\end{split}
\end{equation*}
\end{itemize}

\end{ex}

\subsection{Behavior of $\fQ$ across cancellations}
\label{sec:cancellation}

In this section we prove the following proposition.

\begin{prop}
Let $(t_j, x_j)$ be a cancellation point.  Then
\begin{equation*}
\fQ(t_j+) \leq \fQ(t_j-).
\end{equation*}
\end{prop}

\begin{proof}
Assume that the left incoming wavefront carries the jump $(a,b)$, the right incoming wavefront carries the jump $(b,c)$ and $a<c<b$, as in Figure \ref{fig:canc}. Thanks to Lemma \ref{lem:scambio:onde:stati} and Remark \ref{rem:scambio:onde:stati}, we can assume w.l.o.g. that
\begin{equation*}
\begin{cases}
(a,b] & \text{is the set of left incoming positive waves} \\
(a,c] & \text{is the set of positive waves which survive after the cancellation}.
\end{cases}
\end{equation*}
Assume also that after the cancellation a splitting of the surviving waves occurs and $(a,c]$ is split in
\begin{equation*}
(a,c] := (a, d_1] \cup (d_1, d_2] \cup \cdots \cup (d_{n-1}, c].
\end{equation*}

\begin{figure}
\begin{tikzpicture}[scale=0.8]
\draw [->] (0,1) to (14,1);
\draw [->] (0,1) to (0,8);
\node[right] at (14,1) {$u$};
\node[above] at (0,8) {$F_\e (u)$};
\draw[green] (1,3) to (13,1);
\draw[blue] (1,3) to (6,2.5) to (8,3) to (10,5) to (11,7);
\draw (11,7) to (13,1);
\draw (4,1.1) to (4,0.9);
\draw (8.8,1.1) to (8.8,0.9);
\node[below] at (4,0.9) {$w$};
\node[below] at (8.8,0.9) {$w'$};
\draw[|-|] (3,1.5) to (6,1.5);
\node[above] at (4.5, 1.5) {\small $\jleft(t_j+,w,w')$};
\draw[|-|] (8,1.5) to (9.6,1.5);
\node[above] at (8.8, 1.5) {\small $\jright(t_j+,w,w')$};
\draw (1,3) to (2.6,4) to (3,5) to (3.2,5.2) to (4,3) to (5,3.2) to (6,2.5) to (7,3.4) to (8,3) to (8.6,6) to (9.6,6);
\draw[red] (3,5) to (4,3)  to (6,2.5);
\draw[red] (8,3) to (9.6,6);
\draw (9.6,6) to (11,7);
\node[red,right] at (3.5,4) {\small $\sigmaent(w, \jleft(t_j+,w,w'))$};
\node[red,left] at (9,5) {\small $\sigmaent(w', \jright(t_j+,w,w'))$};
\draw (1,1.1) to (1,0.9);
\draw (11,1.1) to (11,0.9);
\draw (13,1.1) to (13,0.9);
\node[below] at (1,0.9) {$a$};
\node[below] at (11,0.9) {$c$};
\node[below] at (13,0.9) {$b$};
\draw[<->] (11,0.3) to (13,0.3);
\node[below] at (12,0.3) {\tiny waves canceled};
\end{tikzpicture}
\caption{Cancellation: the \textbf{black line} is the graph of $F_\e$; the \textcolor{green}{green line} is the convex envelop of $F_\e$ \textcolor{green}{before the cancellation} (a single shock), the \textcolor{blue}{blue curve} is the convex envelope \textcolor{blue}{after the cancellation} (a splitting is possible); the function $F_\e$ stays above the blue curve; the \textcolor{red}{red lines} are the \textcolor{red}{convex envelopes} of $F_\e$ corresponding to $\jleft(t_j+,w,w')$, $\jright(t_j,w,w')$ respectively.}
\label{F:cancellation:feff}
\end{figure}
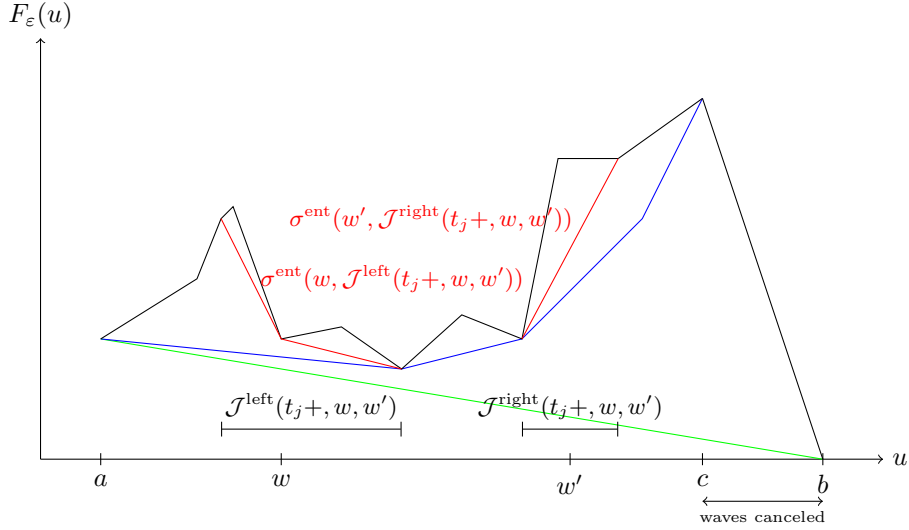

Fix $w,w' \in \W(t_j+)$. We can assume that $w' \in (a,c]$, otherwise it is easy to see that the weight $\mathfrak q(t_j\pm, w,w')$ does not increase. We distinguish two cases.
\begin{itemize}
\item Assume $w \leq a$ and $w' \in (d_{i-1},d_i]$. The denominator $d(t_j \pm,w,w')$ does not change across the interaction. Also the interval $\jleft(t_j\pm,w,w')$ does not change. On the other side, the interval $\jright(t_j\pm,w,w')$ can change across the interaction because of the splitting due to the cancellation, but, however
\begin{equation*}
\sup \jright(t_j+,w,w') = \sup \jright(t_j-,w,w') \leq c,
\end{equation*}
while
\begin{equation*}
\inf \jright(t_j-, w,w') = a, \qquad
\inf \jright(t_j+,w,w') = d_{i-1}.
\end{equation*}
Notice that
\begin{equation*}
\conv_{[a,c]} F_\e(d_{i-1}) = F_\e(d_{i-1})
\end{equation*}
and thus also 
\begin{equation*}
\conv_{\jright(t_j-,w,w')} F_\e(d_{i-1}) = F_\e(d_{i-1}).
\end{equation*}
Hence 
\begin{equation*}
\sigmaent(w', \jleft(t_j-,w,w')) = \sigmaent(w', \jleft(t_j+,w,w')),
\end{equation*}
thus $\pi(t_j-,w,w') = \pi(t_j+,w,w')$ and $\mathfrak q(t_j-, w,w') = \mathfrak q(t_j+,w,w')$. 

\item Assume now $w,w' \in (a,c]$. In this case, before the cancellation, since $w,w'$ have the same position, $\mathfrak q(t_j-,w,w') = 0$.  After the cancellation the situation is as in Figure \eqref{F:cancellation:feff}; since the splitting is obtained taking the convex envelope on the surviving waves (see Figure \ref{F:cancellation:feff}), the numerator in \eqref{E:weight} after the cancellation is zero and thus also $\mathfrak q(t_j+,w,w') = 0$. 
\end{itemize}

This concludes the proof of the proposition.
\end{proof}

\begin{rem}
\label{rem:splitting}
The splitting caused by a cancellation is a typical behavior of conservation laws when the flux function $F$ is not convex: a very small cancellation can split a big shock into small pieces.

Such splitting phenomenon is the reason why we need to define the denominator $d(\bar t,w,w')$ of $\mathfrak q(\bar t,w,w')$ looking at the profile of the solution $u^\e$ at the future time $\tint(\bar t,w,w')$ and not only at the present profile $u^\e(\bar t)$. 


\end{rem}

\subsection{Behavior of $\fQ$ across same-sign interactions}
\label{sec:interaction}

Let now $(t_j, x_j)$ be a same-sign interaction point. Assume that the left incoming wavefront carries the jump $(a,b)$, while the right incoming wavefront carries the jump $(b,c)$, with $a<b<c$, , as in Figure \ref{fig:ssint}. Set $s':= b-a$, $s'':=c-b$. Denote by $\sigma', \sigma''$ the speed of the two incoming wavefronts, respectively. After the interaction the outgoing wavefront has strength $c-a = s'+s''$ and speed
\begin{equation*}
\sigma := \frac{\sigma' s' + \sigma'' s'' }{s'+s''}.
\end{equation*}
The change in speed at $(t_j,x_j)$ is thus:
\begin{equation}
\label{E:quadr:aoi}
\Delta \sigma(t_j,x_j) := |\sigma - \sigma'||s'| + |\sigma - \sigma''||s''| = 2 \cdot  \frac{\sigma' - \sigma''}{s'+s''}s's''.
\end{equation}
Recall that $\sigma'>\sigma''$ because the two wavefronts are interacting. 
\begin{prop}
\label{prop:interaction}
Let $(t_j, x_j)$ be a same-sign interaction point. Then
\begin{equation*}
\frac{1}{2} \Delta \sigma(t_j,x_j) \leq \fQ(t_j-) - \fQ(t_j+).
\end{equation*}
\end{prop}
\begin{proof}
Thanks to Lemma \ref{lem:scambio:onde:stati} and Remark \ref{rem:scambio:onde:stati}, we can assume w.l.o.g. that
\begin{equation*}
\begin{cases}
(a,b] & \text{is the set of left incoming waves}, \\
(b,c] & \text{is the set of right incoming waves}.
\end{cases}
\end{equation*}

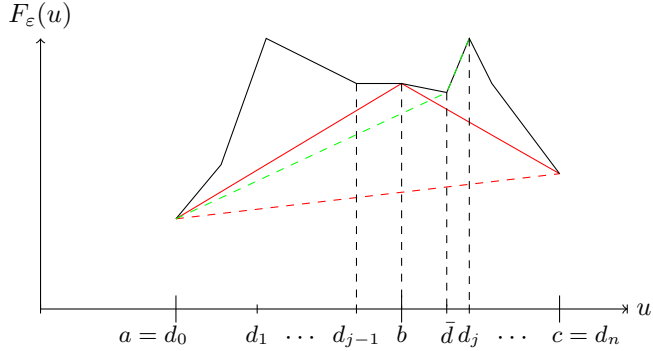
\begin{figure}
\begin{tikzpicture}[scale=0.6]
\draw[->] (0,0) to (13,0);
\draw[->] (0,0) to (0,6);
\node[right] at (13,0) {$u$};
\node[above] at (0,6) {$F_\e(u)$};
\draw[red] (3,2) to (8,5) to (11.5,3);
\draw[red, dashed] (3,2) to (11.5,3);
\draw (3,2) to (4,3.2) to (5,6) to (7,5) to (8,5) to (9,4.8) to (9.5,6) to (10,5) to (11.5,3);
\draw (0,0.1) to (0,-0.1);
\draw (13,0.1) to (13,-0.1);
\draw (3,0.3) to (3,-0.3);
\draw (8,0.3) to (8,-0.3);
\draw (11.5,0.3) to (11.5,-0.3);
\draw (4.8,0.1) to (4.8,-0.1);
\draw (7,0.1) to (7,-0.1);
\draw (9.5,0.1) to (9.5,-0.1);
\draw (9,0.1) to (9,-0.1);
\draw[dashed, green] (3,2) to (9,4.8) to (9.5,6);
\draw[dashed] (7,5) to (7,0);
\draw[dashed] (8,5) to (8,0);
\draw[dashed] (9,4.8) to (9,0);
\draw[dashed] (9.5,6) to (9.5,0);
\node[below] at (2.5,-0.15) {\small $a = d_0$};
\node[below] at (8,-0.15) {\small $b$};
\node[below] at (12.1,-0.15) {\small $c = d_n$};
\node[below] at (4.8,-0.15) {\small $d_1$};
\node[below] at (7,-0.15) {\small $d_{j-1}$};
\node[below] at (9,-0.1) {\small $\bar d$};
\node[below] at (9.5,-0.15) {\small $d_{j}$};
\node[below] at (5.8, -0.3) {$\cdots$};
\node[below] at (10.5, -0.3) {$\cdots$};
\end{tikzpicture}
\caption{Interaction (same sign): the two interacting wavefronts are $(a,b]$ on the left and $(b,c]$ on the right. The \textbf{black line} is the graph of $F_\e$. The \red{red lines} are the convex envelopes of $F_\e$ before and after (dashed line) the interaction. The \green{green dashed line} is the convex envelope of $F_\e$ on $[a,d_{j-1}]$, used to compute $\pi(t_j+,w,w')$ for waves in $(d_{j-1}, d_j]$ after the interaction.}
\label{F:interaction}
\end{figure}

The weight $\mathfrak q(t_j\pm,w,w')$ can change at $t_j$ only if at least one among $w,w'$ is in $(a,c]$. Therefore is it sufficient to prove that:
\begin{enumerate}
\item the integral on pairs $(w,w')$ with both  $w,w' \in (a,c]$ bounds the change in speed:
\begin{equation*}
\int_a^b \int_b^c \Big[ \mathfrak q(t_j-,w,w') - \mathfrak q(t_j+,w,w')\Big] dw'dw \geq \frac{1}{2} \Delta \sigma(t_j, x_j);
\end{equation*}

\item the integral on pairs $(w,w')$ with only one between $w,w'$ in $(a,c]$ (say $w'$) is decreasing:
\begin{equation*}
\int^a \int_a^c \Big [\mathfrak q(t_j+,w,w') - \mathfrak q(t_j-,w,w') \Big] dw'dw \leq 0.  
\end{equation*}
(The computation for $\int_a^c \int_c [ \mathfrak q(t_j+) - \mathfrak q(t_j-) ] dw'dw$ is analogous). 
\end{enumerate}

For 1. it is  sufficient to observe that whenever $(w,w') \in (a,b] \times [b,c]$,  $\mathfrak q(t_j+,w,w') = 0$, while 
\begin{equation*}
\mathfrak q(t_j-,w,w') = \frac{\sigma' - \sigma''}{s'+s''},
\end{equation*}
because, by definition, $\jleft(t_j-,w,w') = (a,b]$ and $\jright(t_j-,w,w') = (b,c]$ and the convex envelopes on $(a,b]$ and $(b,c]$ are given by single segments. 

For 2. we proceed as follows. Fix $w \in (0,a]$. Let $$d_0 = a < d_1 < \cdots < d_{n-1} < d_n = c,$$ such that for $w',w'' \in (d_{i-1}, d_{i}]$, $\tint(t_j,w,w') = \tint(t_j,w,w'')$. See Figure \ref{F:interaction}.  Notice that if $w' \in (d_{i-1}, d_{i}]$,
\begin{equation*}
\jright(t_j+,w,w') = (a,d_{i}], 
\end{equation*}
while
\begin{equation*}
\jright(t_j- ,w,w') =
\begin{cases}
(a, \min\{b,d_{i}\}] & \text{if } w' \in (a,b], \\
(b, d_{i}] & \text{if } w' \in (b,c]. 
\end{cases}
\end{equation*}
Therefore the weight $\mathfrak q(t_j \pm,w,w')$ can increase only if $w' \in (d_{j-1}, d_{j}] \cap (a,b]$, where $j$ is the unique index such that $d_{j-1} \leq b < d_{j}$. Even more precisely, let 
\begin{equation}
\label{eq:schock:int}
(a,\bar d) := \bigcup \Big \{I \subseteq \R \text{ interval} \ \Big| \ b \in I, \ \conv_{[a,d_j]} F_\e < F_\e \text{ on } I \Big\}
\end{equation}
be the maximal interval such that $\conv_{[a,\bar d]} F_\e < F_\e$ on $(a,\bar d)$.
%
If $w' \in (\bar d, d_j]$, then $\mathfrak q(t_j,w,w')$ does not change, because 
\begin{equation*}
\conv_{[a,d_j]} F_\e(\bar d) = \conv_{[b,d_j]} F_\e(\bar d) = F_\e(\bar d).
\end{equation*}
%
Therefore the weight $\mathfrak q(t_j \pm,w,w')$ can increase only if $w' \in  (d_{j-1}, b]$. However, in this case the weight of the waves in $(b, \bar d]$ decreases and, roughly speaking, such decrease is enough to compensate the increase of the weight of the waves in $(d_{j-1}, b]$. This can be proved as follows. 

First of all observe that for any $w' \in (d_{j-1}, \bar d]$, the denominator does not change:
\begin{equation*}
d(t_j-, w,w') = d(t_j+,w,w').
\end{equation*}
Therefore it is sufficient to compute the change of $\fQ$ due to the numerator. Notice that at time $t_j+$ it holds
\begin{equation*}
\jright(t_j+,w,w') = (a, d_{j}];
\end{equation*}
therefore, by Lemma \ref{L:fund:prop}, the interval $\jleft: = \jleft(t_j+,w,w')$ does not depend on $w' \in (d_{j-1}, \bar d]$. Moreover, by definition, 
\begin{equation*}
\jleft(t_j-,w,w') = \jleft(t_j+,w,w') = \jleft.
\end{equation*}
Notice also that, by the properties of the convex envelopes,
\begin{equation*}
\sigmaent \big(w', (a, d_j] \big) = \sigmarh \big((a,\bar d]  \big), \qquad
\sigmaent \big(w', (b, d_j] \big) = \sigmaent \big(w', (b,\bar d]  \big).
\end{equation*}
Hence we have
\begin{equation*}
\begin{split}
\pi(t_j+,w,w') & = 
\Big[\sigmaent(w, \jleft) - \sigmaent \big(w', (a,d_{j}] \big) \Big]^+ \\
& = 
\Big[\sigmaent(w, \jleft) - \sigmarh \big( (a,\bar d] \big) \Big]^+
\end{split}
\end{equation*}
and
\begin{equation*}
\pi(t_j-,w,w') = 
\begin{cases}
\Big[\sigmaent(w, \jleft) - \sigmarh \big((a,b] \big) \Big]^+ & \text{if } w' \in (d_{j-1}, b], \\
\Big[\sigmaent(w, \jleft) - \sigmaent \big(w', (b,\bar d] \big) \Big]^+
& \text{if } w' \in (b, d_{j}].
\end{cases}
\end{equation*}
Therefore for the numerator it holds (recall that $w \leq a$ is fixed)
\begin{equation*}
\begin{split}
\int_{d_{j-1}}^{\bar d} & \pi(t_j+,w,w') dw' \\
& =  \int_{d_{j-1}}^{\bar d} \Big[\sigmaent(w, \jleft) - \sigmarh\big((a, \bar d]\big) \Big]^+ dw' \\
& = 
\Big[\sigmaent(w, \jleft) - \sigmarh\big((a, \bar d]\big) \Big]^+ (\bar d - d_{j-1}) \\
& = 
\Big[\sigmaent(w, \jleft) - \sigmarh\big( (a, \bar d]\big) \Big]^+ (\bar d - a) \\
& \qquad \qquad - 
\Big[\sigmaent(w, \jleft) - \sigmarh\big( (a, \bar d]\big) \Big]^+ (d_{j-1} - a) \\
& = 
\Big[\sigmaent(w, \jleft)(\bar d - a) - \sigmarh\big( (a, \bar d]\big) (\bar d - a) \Big]^+ \\
& \qquad \qquad
- \Big[\sigmaent(w, \jleft) - \sigmarh\big( (a, \bar d]\big) \Big]^+ (d_{j-1} - a) \\
& = 
\Big[\sigmaent(w, \jleft)(\bar d - a) - \int_a^{b} \sigmarh \big((a, b]\big)dw' - \int_b^{\bar d} \sigmaent \big(w', (b, \bar d] \big) dw' \Big]^+ \\
& \qquad \qquad 
- 
\Big[\sigmaent(w, \jleft) - \sigmarh\big((a, \bar d]\big) \Big]^+ (d_{j-1} - a) \\
& = 
\Bigg[\int_a^{b} \Big( \sigmaent(w, \jleft) - \sigmarh \big( (a, b]\big) \Big)dw' \\
& \qquad \qquad
+ \int_b^{\bar d} \Big( \sigmaent(w, \jleft) - \sigmaent \big(w', (b, \bar d] \big) \Big) dw' \Bigg]^+ \\
& \qquad \qquad \qquad \qquad 
- 
\Big[\sigmaent(w, \jleft) - \sigmarh\big((a, \bar d]\big) \Big]^+ (d_{j-1} - a) \\
& \leq
\int_a^{b} \Big[ \sigmaent(w, \jleft) - \sigmarh \big( (a, b]\big) \Big]^+dw' \\
& \qquad \qquad 
+ \int_b^{\bar d} \Big[ \sigmaent(w, \jleft) - \sigmaent \big(w', (b, \bar d] \big) \Big]^+ dw'  \\
& \qquad \qquad 
- 
\Big[\sigmaent(w, \jleft) - \sigmarh\big((a, \bar d]\big) \Big]^+ (d_{j-1} - a) \\
\text{(since } & \text{$\sigmarh \big((a,b] \big) \geq \sigmarh \big((a,\bar d] \big)$, being $(a, \bar d]$ defined through \eqref{eq:schock:int})} \\
& \leq
\int_{d_{j-1}}^{b} \Big[ \sigmaent(w, \jleft) - \sigmarh \big( (a, b]\big) \Big]^+dw' \\
& \qquad \qquad
+ \int_b^{\bar d} \Big[ \sigmaent(w, \jleft) - \sigmaent \big(w', (b, \bar d] \big) \Big]^+ dw',  \\
\end{split}
\end{equation*}
\begin{equation*}
\begin{split}
\text{(by } & \text{the properties of the convex envelopes and the definition of $\bar d$)} \\
& =
\int_{d_{j-1}}^{b} \Big[ \sigmaent(w, \jleft) - \sigmarh \big( (a, b]\big) \Big]^+dw' \\
& \qquad \qquad
+ \int_b^{\bar d} \Big[ \sigmaent(w, \jleft) - \sigmaent \big(w', (b, d_j] \big) \Big]^+ dw',  \\
& = 
\int_{d_{j-1}}^{b} \pi(t_j-,w,w') dw' + \int_b^{\bar d} \pi(t_j-,w,w') dw' \\
& = \int_{d_{j-1}}^{\bar d} \pi(t_j-,w,w') dw',
\end{split}
\end{equation*}
which concludes the proof of Proposition \ref{prop:interaction} and therefore also the proof of Proposition \ref{prop:main}. 
\end{proof}

\section{Extension to systems}
\label{sec:extension}

The extension of Theorem \ref{thm:main} to the case of systems \eqref{eq:cauchy} is pretty technical. However, it requires only the same ideas and techniques already used in the paper \cite{BiaMod3}, where a quadratic interaction potential (depending also on the past profiles $u^\e(t)$, $t < \bar t$) for systems is constructed. In this section we present a brief (and far from complete) overview of such ideas and techniques. For details we refer to the cited technical article \cite{BiaMod3} and to the notes \cite{Mod} (for a more ``user-friendly'' exposition). 

In order to extend Theorem \ref{thm:main} to the case of systems \eqref{eq:cauchy}, three main issues have to be addressed.
\begin{enumerate}

\item The flux used to solve a Riemann problem $(u^L, u^R)$ is not the flux $F$ (or its approximate version $F_\e$), as in the scalar case; on the contrary, $n$ \emph{reduced fluxes} $\tilde f_k$ ($k=1,\dots, n$, one for each family) are constructed, which depend on the left and right states $u^L, u^R$.

\item Besides same-sign interactions and cancellations,  also \emph{transversal interactions}, i.e. interactions between wavefronts of different families, have to be considered.

\item When an interaction between wavefronts takes place, there can be a \emph{small creation or cancellation} of waves, due to the non-linearity of the flux. For instance, an interaction between two wavefronts of the first family can produce a small creation of waves of the first family and also a small creation of waves of the other families. 


\end{enumerate}

These three problems are solved as follows. First of all we construct a suitable wave tracing, as in Section \ref{sec:wavetracing}. Here, however, we introduce $n$ sets of waves $\W_k$ ($k=1, \dots, n$, one for each family) and for each wave $w \in \W_k$, we define its sign $\mathcal S(w)$, its creation time $\tcr(w)$ (which can be strictly greater than $0$) and its cancellation time $\tcanc(w)$, its position $\mathtt X_k(t,w)$  and its speed $\sigma_k(t,w)$ at each time $t \in [\tcr(w), \tcanc(w))$.

\subsection{The effective fluxes}

The first issue, i.e. the fact that in the solution to a Riemann problem one has to use the reduced fluxes $\tilde f_k$, $k=1, \dots, n$ and not the original flux $F$, is solved through the introduction, for each time $t$, of $n$ \emph{effective fluxes}
\begin{equation*}
\feff_k(t): \W_k(t) \to \R \quad (k=1, \dots, n), \qquad \feff_k(t) \in C^{1,1},
\end{equation*}
where $\W_k(t) := \{w \in \W_k \ | \ \tcr(w) \leq t < \tcanc(w) \}$. 

The key property of the effective fluxes is the following: assume that at point $\bar x$ the piecewise constant map $x \mapsto u^\e(t, x)$ has a jump between $u^L$ and $u^R$; then, for each $k=1, \dots, n$, the effective flux $\feff_k(\bar t)$, restricted to the set $\mathtt X_k(\bar t)^{-1}(\bar x)$,  coincides, up to affine functions, with the reduced flux $\tilde f_k$ used to solve the Riemann problem $(u^L, u^R)$. 

Now, as in Section \ref{sec:def:q}, given $w,w' \in \W_k(\bar t)$, $w<w'$, with the same sign, different positions and such that they will interact after time $\bar t$, we can define $\W_k(\bar t,w,w')$ as in \eqref{eq:Wtww}, the intervals $\jleft_k, \jright_k$ as in \eqref{E:jj}, the weight $\mathfrak q_k(t,w,w')$ as in \eqref{E:weight}, the numerator $\pi_k(\bar t,w,w')$ as in \eqref{eq:pi} and the denominator $d_k(\bar t,w,w')$ as in \eqref{eq:di}. The only difference now is that  $\sigmaent(w,\jleft_k)$ (resp. $\sigmaent(w', \jright_k)$) is the entropic speed given to the wave $w$ (resp. $w'$) by the interval $\jleft_k$ (resp. $\jright_k$) and the effective flux $\feff_k(\bar t)$.

\subsection{Transversal interactions}

Transversal interactions generate two problems.

 First of all, the waves involved in a transversal interaction can change their speed because of the transversal interaction. However such change of speed is bounded by the decrease of the standard transversal interaction potential (introduced by Glimm in \cite{gli65}):
\begin{equation*}
\Qtrans(t) := \sum_{k < h} \mathcal L^2 \Big( \Big\{ (w,w') \in \W_k \times \W_h \ \Big| \ \mathtt X_k(t,w) > \mathtt X_h(t,w') \Big\} \Big). 
\end{equation*}

Secondly,  the weights $\mathfrak q_k(t,w,w')$ can change across transversal interactions. More precisely, the denominator $d_k(t,w,w')$ does not change, by definition. On the contrary, the numerator $\pi_k(t,w,w')$ can change (and possibly increase), for two reasons:
\begin{itemize}
\item in general the effective flux $\feff_k(t-)$ before the interaction is different from the effective flux $\feff_k(t+)$ after the interaction; this can cause an increase of the numerator $\pi_k(t,w,w')$; however the change of the effective flux due to the transversal interaction is bounded, in a suitable norm, by the decrease of $\Qtrans$; we will thus substitute $\fQ$ (which is not decreasing anymore) with $\fQ + c \Qtrans$, where $c$ is a sufficiently big constant;
\item as in the case of cancellations (Section \ref{sec:cancellation}), transversal interactions can cause  splittings; such splittings, however, can be treated exactly as in the case of cancellations, analyzing separately the case when both $w,w'$ are involved in the transversal interaction and the case when at most one of them is involved in the transversal interaction. 
\end{itemize}

\subsection{Small creations and cancellations}

As for transversal interactions, also small creations and cancellations generate two problems. First of all, the waves can change their speed due to these small creations/cancellations. However such change of speed is bounded by the decrease of the cubic potential introduced by Bianchini in \cite{Bia}:
\begin{equation*}
\Qcubic(t): = \sum_{k=1}^n \iint_{\substack{w,w' \in \W_k(t) \\ w<w'}} \big|\sigma_k(t,w) - \sigma_k(t,w') \big| dwdw'.
\end{equation*}

Secondly, the weights $\mathfrak q_k(t,w,w')$ can change (and possibly increase) due to these small creations/cancellations. However, also in this case, it is not  difficult to prove that the change of $\mathfrak q_k$ is bounded by the decrease of $\Qcubic$. 

Therefore, in order to guarantee that the interaction potential $\Upsilon$ is decreasing also in the case of system, we have to substitute definition \eqref{eq:upsilon} with the following:
\begin{equation*}
\Upsilon(t) := \fQ(t) + c \big[  \TV(u^\e(t)) +  \Qtrans(t) + \Qcubic(t) \big],
\end{equation*}
for a sufficiently big constant $c>0$. In particular:
\begin{itemize}
\item $\fQ(t)$ bounds the change of speed due to same-sign interactions;
\item $\TV(u^\e(t))$ bounds the chance of speed due to cancellations;
\item $c \Qtrans(t)$ bounds the change of speed and the possible increase of $\TV(u^\e(t))+ \fQ(t)$ due to transversal interactions;
\item $c \Qcubic(t)$ bounds the change of speed and the possible increase of $\TV(u^\e(t))+ \fQ(t)$ due to small creations/cancellations.
\end{itemize}

\end{document}